\documentclass[a4paper,12pt]{amsart}
\usepackage{amsfonts,amsmath,amsthm,amssymb}
\usepackage{enumitem}
\usepackage[margin=1in]{geometry}
\usepackage[colorlinks=true,linkcolor=blue,filecolor=blue,citecolor=red]{hyperref}
\usepackage{cleveref}

\theoremstyle{plain}
\newtheorem{theorem}{Theorem}[section]
\newtheorem{corollary}[theorem]{Corollary}
\newtheorem{lemma}[theorem]{Lemma}

\theoremstyle{definition}

\begin{document}

\title[Some congruences for $3$-core cubic bipartitions]{Some congruences for $3$-core cubic bipartitions}
\author[Russelle Guadalupe]{Russelle Guadalupe}
\address{Institute of Mathematics, University of the Philippines-Diliman\\
	Quezon City 1101, Philippines}
\email{rguadalupe@math.upd.edu.ph}

\renewcommand{\thefootnote}{}



\renewcommand{\thefootnote}{\arabic{footnote}}
\setcounter{footnote}{0}

\begin{abstract}
In this paper, we derive certain congruences for the number of $3$-core cubic bipartitions using elementary $q$-series manipulations and dissection formulas.
\end{abstract}

\maketitle

\section{Introduction}
Throughout this paper, we denote $(q^m;q^m)_\infty = \prod_{n\geq 1}(1-q^{mn})$ for $|q| < 1$ and positive integer $m$. Recall that a partition of a positive integer $n$ is a nonincreasing sequence of positive integers, known as its parts, whose sum is $n$. The study of the arithmetic properties of the number $p(n)$ of partitions of $n$ was initiated by Ramanujan \cite{raman,raman2}, who proved the following remarkable congruences
\begin{align*}
p(5n+4)&\equiv 0\pmod{5},\\
p(7n+5)&\equiv 0\pmod{7},\\
p(11n+6)&\equiv 0\pmod{11}.
\end{align*}
Ahlgren \cite{ahlg} generalized this result by showing that infinitely many congruences of the form $p(An+B)\equiv 0\pmod{M}$ exist when $M$ is coprime to $6$. Since the inception of $p(n)$, analogous congruences for the other types of partitions have been established by various authors \cite{anhrse,chernda,cuigu,naika,wang,wang2}, including those obtained by Chan \cite{chan,chan2} for the cubic partition function $a(n)$, with generating function given by 
\begin{align*}
\sum_{n=0}^\infty a(n)q^n =\dfrac{1}{(q;q)_\infty(q^2;q^2)_\infty},
\end{align*}
and by Chen \cite{chen,chen2} for the $t$-core partition function $c_t(n)$, with generating function given by
\begin{align*}
\sum_{n=0}^\infty c_t(n)q^n =\dfrac{(q^t;q^t)_\infty^t}{(q;q)_\infty}
\end{align*}
for $t\geq 1$. We note that $a(n)$ is the number of partitions in two colors such that the second color only appears in even parts, and $a_t(n)$ is the number of partitions whose Ferrers graph has no hook numbers divisible by $t$.
In 2017, Gireesh \cite{giree} examined the arithmetic properties of the function $C_3(n)$, which is the number of cubic partitions of $n$ which are also $3$-cores, whose generating function is given by
\begin{align*}
\sum_{n=0}^\infty C_3(n)q^n =\dfrac{(q^3;q^3)_\infty^3(q^6;q^6)_\infty^3}{(q;q)_\infty(q^2;q^2)_\infty}.
\end{align*}	
Recently, da Silva and Sellers \cite{dasilse} extended Gireesh's results by obtaining dissections and congruences for $C_3(n)$, including $C_3(24n+21)\equiv 0\pmod{4}$ and $C_3(18n+14)\equiv 0\pmod{9}$. Motivated by their results, we derive in this paper some congruences for the number $CP_3(n)$ of $3$-core cubic bipartitions, whose generating function is given by
\begin{align*}
\sum_{n=0}^\infty CP_3(n)q^n =\dfrac{(q^3;q^3)_\infty^6(q^6;q^6)_\infty^6}{(q;q)_\infty^2(q^2;q^2)_\infty^2}.
\end{align*}
More precisely, $CP_3(n)$ is the number of ordered pairs $(\mu,\kappa)$ such that both $\mu$ and $\kappa$ are $3$-core cubic partitions such that the sum of all of the parts of $\mu$ and $\kappa$ is $n$. We apply elementary $q$-series manipulations and dissection formulas to present our results.

We organize our paper as follows. In Section \ref{sec2}, we list some identities that are needed to prove our results. Section \ref{sec3} contains certain dissection formulas and resulting congruences for $CP_3(n)$.

\section{Preliminaries}\label{sec2}

In this section, we present several $2$- and $3$-dissection formulas needed to derive our congruences involving $CP_3(n)$. We denote $f_m := (q^m;q^m)_{\infty}$ for integer $m\geq 1$ throughout the rest of the paper.

\begin{lemma}\label{lem21} The following $2$-dissections hold:
\begin{align}
f_1^2 &= \dfrac{f_2f_8^5}{f_4^2f_{16}^2}-2q\dfrac{f_2f_{16}^2}{f_8},\label{eq1}\\
\dfrac{1}{f_1^2} &= \dfrac{f_8^5}{f_2^5f_{16}^2}+2q\dfrac{f_4^2f_{16}^2}{f_2^5f_8},\label{eq2}\\
f_1^4 &= \dfrac{f_4^{10}}{f_2^2f_8^4}-4q\dfrac{f_2^2f_8^4}{f_4^2},\label{eq3}\\
\dfrac{f_3^3}{f_1} &= \dfrac{f_4^3f_6^2}{f_2^2f_{12}}+q\dfrac{f_{12}^3}{f_4},\label{eq4}\\
\dfrac{f_3^2}{f_1^2} &= \dfrac{f_4^4f_6f_{12}^2}{f_2^5f_8f_{24}}+2q\dfrac{f_4f_6^2f_8f_{24}}{f_2^4f_{12}},\label{eq5}\\
\dfrac{f_3}{f_1^3} &= \dfrac{f_4^6f_6^3}{f_2^9f_{12}^2}+3q\dfrac{f_4^2f_6f_{12}^2}{f_2^7},\label{eq6}\\
\dfrac{1}{f_1f_3} &= \dfrac{f_8^2f_{12}^5}{f_2^2f_4f_6^4f_{24}^2}+q\dfrac{f_4^5f_{24}^2}{f_2^4f_6^2f_8^2f_{12}}.\label{eq7}
\end{align}
\end{lemma}

\begin{proof}
We refer to \cite[Lem. 1]{dasilse} for a proof of (\ref{eq3}) and \cite[p. 40, Entry 25]{berndt} for proofs of (\ref{eq1}) and (\ref{eq2}). Identities (\ref{eq4}), (\ref{eq5}) and (\ref{eq7}) follow from \cite[(22.7.5)]{hirsc}, \cite[(30.9.9)]{hirsc} and \cite[(30.12.3)]{hirsc}, respectively. Identity (\ref{eq6}) follows after replacing $q$ with $-q$ in \cite[(27.7.3)]{hirsc}.
\end{proof}

\begin{lemma}\label{lem22}The following $3$-dissections hold:
\begin{align}
\dfrac{1}{f_1f_2} &= \dfrac{f_9^9}{f_3^6f_6^2f_{18}^3}+q\dfrac{f_9^6}{f_3^5f_6^3}+3q^2\dfrac{f_9^3f_{18}^3}{f_3^4f_6^4}-2q^3\dfrac{f_{18}^6}{f_3^3f_6^5}+4q^4\dfrac{f_{18}^9}{f_3^2f_6^6f_9^3},\label{eq8}\\
\dfrac{f_2}{f_1^2} &= \dfrac{f_6^4f_9^6}{f_3^8f_{18}^3}+2q\dfrac{f_6^3f_9^3}{f_3^7}+4q^2\dfrac{f_6^2f_{18}^3}{f_3^6}.\label{eq9}
\end{align}
\end{lemma}

\begin{proof}
Identity (\ref{eq8}) follows from \cite[(13)]{chan}. To obtain (\ref{eq9}), replace $q$ with $q^{1/3}$ in the $q$-expansion of $\frac{f_2}{f_1^2}$ and use \cite[Lem. 2.1, 2.6-2.8]{anhrse}. The desired identity now follows from replacing $q$ with $q^3$.
\end{proof}

\begin{lemma}\label{lem23} We have  
\begin{align}
\dfrac{f_1f_2}{f_9f_{18}} &= \dfrac{1}{x(q^3)}-q-2q^2x(q^3),\label{eq10}\\
\dfrac{f_3^4f_6^4}{f_9^4f_{18}^4} &= \dfrac{1}{x(q^3)^3}-7q^3-8q^6x(q^3)^3,\label{eq11}
\end{align}
where $x(q)=q^{-1/3}c(q)$, and 
\begin{align*}
c(q)=\dfrac{q^{1/3}}{1+\dfrac{q+q^2}{1+\dfrac{q^2+q^4}{1+\cdots}}}
\end{align*}
is the Ramanujan's cubic continued fraction.
\end{lemma}

\begin{proof}
See \cite[Lem. 2.2-2.3]{giree}.
\end{proof}

\section{Dissections and congruences for $CP_3(n)$}\label{sec3}

We start presenting our results by giving a $2$-dissection of the generating function 
\begin{align}
\sum_{n=0}^\infty CP_3(n)q^n =\dfrac{f_3^6f_6^6}{f_1^2f_2^2}.\label{eq12}
\end{align}

\begin{theorem}\label{thm31} We have
\begin{align}
\sum_{n=0}^\infty CP_3(2n)q^n &= \dfrac{f_2^6f_3^{10}}{f_1^6f_6^2}+q\dfrac{f_3^6f_6^6}{f_1^2f_2^2},\label{eq13}\\
\sum_{n=0}^\infty CP_3(2n+1)q^n &= 2\dfrac{f_2^2f_3^8f_6^2}{f_1^4}.\label{eq14}
\end{align}
\end{theorem}

\begin{proof}
Substituting (\ref{eq4}) into (\ref{eq12}) yields
\begin{align*}
\sum_{n=0}^\infty CP_3(n)q^n = \left(\dfrac{f_3^3}{f_1}\right)^2\dfrac{f_6^6}{f_2^2}=\left(\dfrac{f_4^3f_6^2}{f_2^2f_{12}}+q\dfrac{f_{12}^3}{f_4}\right)^2\dfrac{f_6^6}{f_2^2}.
\end{align*}
Extracting the terms with odd and even exponents gives 
\begin{align*}
\sum_{n=0}^\infty CP_3(2n)q^{2n} &= \dfrac{f_4^6f_6^{10}}{f_2^6f_{12}^2}+q^2\dfrac{f_6^6f_{12}^6}{f_2^2f_4^2},\\
\sum_{n=0}^\infty CP_3(2n+1)q^{2n+1} &= 2q\dfrac{f_4^2f_6^8f_{12}^2}{f_2^4}.
\end{align*}
We then divide the last identity by $q$ and replace $q^2$ to $q$, arriving at identities (\ref{eq13}) and (\ref{eq14}).
\end{proof}

We next give congruences modulo $8$ and $16$ for $CP_3(n)$. 

\begin{theorem}\label{thm32} We have
\begin{align*}
CP_3(8n+3)&\equiv 0\pmod{8},\\
CP_3(8n+7)&\equiv 0\pmod{16}.
\end{align*}
\end{theorem}

\begin{proof}
Substituting (\ref{eq5}) and (\ref{eq3}) into (\ref{eq14}), we have
\begin{align*}
\sum_{n=0}^\infty CP_3(2n+1)q^n &= 2\dfrac{f_2^2f_3^8f_6^2}{f_1^4}\\
&=2\left(\dfrac{f_4^4f_6f_{12}^2}{f_2^5f_8f_{24}}+2q\dfrac{f_4f_6^2f_8f_{24}}{f_2^4f_{12}}\right)^4\left(\dfrac{f_4^{16}}{f_2^2f_8^4}-4q\dfrac{f_2^2f_8^4}{f_4^2}\right)f_2^2f_6^2.
\end{align*}
We extract the terms with odd exponents, divide by $q$, and then replace $q^2$ with $q$ so that
\begin{align}
\sum_{n=0}^\infty CP_3(4n+3)q^n\equiv -8\dfrac{f_2^{14}f_3^6f_6^8}{f_1^{16}f_{12}^4}+16\dfrac{f_2^{23}f_3^7f_6^5}{f_1^{19}f_4^6f_{12}^2}\pmod{32}.\label{eq15}
\end{align}
We next plug (\ref{eq1}) and (\ref{eq6}) into (\ref{eq15}) to get 
\begin{align}
\sum_{n=0}^\infty CP_3(4n+3)q^n&\equiv -8\left(\dfrac{f_4^6f_6^3}{f_2^9f_{12}^2}+3q\dfrac{f_4^2f_6f_{12}^2}{f_2^7}\right)^6\left(\dfrac{f_2f_8^5}{f_4^2f_{16}^2}-2q\dfrac{f_2f_{16}^2}{f_8}\right)\dfrac{f_2^{14}f_6^8}{f_{12}^4}\nonumber\\
&+16\left(\dfrac{f_4^6f_6^3}{f_2^9f_{12}^2}+3q\dfrac{f_4^2f_6f_{12}^2}{f_2^7}\right)^7\left(\dfrac{f_2f_8^5}{f_4^2f_{16}^2}-2q\dfrac{f_2f_{16}^2}{f_8}\right)\dfrac{f_2^{23}f_6^5}{f_4^6f_{12}^2}\pmod{32}.\label{eq16}
\end{align}
Extracting the terms in (\ref{eq16}) with even exponents and replacing $q^2$ with $q$ yields
\begin{align*}
\sum_{n=0}^\infty CP_3(8n+3)q^n&\equiv 8\dfrac{f_2^{34}f_3^{26}f_4^5}{f_1^{39}f_6^{16}f_8^2}-8q\dfrac{f_2^{26}f_6^{22}f_4^5}{f_1^{35}f_6^8f_8^2}-8q^2\dfrac{f_2^{18}f_3^{18}f_4^5}{f_1^{31}f_8^2}+8q^3\dfrac{f_2^{10}f_3^{14}f_4^5f_6^8}{f_1^{27}f_8^2}\pmod{32},
\end{align*}
which shows that $CP_3(8n+3)\equiv 0\pmod{8}$. On the other hand, choosing the terms in (\ref{eq16}) with odd exponents, dividing by $q$, and replacing $q^2$ with $q$ leads us to
\begin{align*}
\sum_{n=0}^\infty CP_3(8n+7)q^n&\equiv -16\dfrac{f_2^{36}f_3^{26}f_8^2}{f_1^{39}f_4f_6^{16}}+16q\left(\dfrac{f_2^{22}f_3^{20}f_4^5}{f_1^{33}f_6^4f_8^2}-\dfrac{f_2^{28}f_3^{22}f_8^2}{f_1^{35}f_4f_6^8}\right)-16q^2\dfrac{f_2^{20}f_3^{18}f_8^2}{f_1^{31}f_4}\\
&+16q^3\left(\dfrac{f_2^6f_3^{12}f_4^5f_6^{12}}{f_1^{25}f_8^2}-\dfrac{f_2^{12}f_3^{14}f_6^8f_8^2}{f_1^{27}f_4}\right),\pmod{32}
\end{align*}
which proves the congruence $CP_3(8n+7)\equiv 0\pmod{16}$.
\end{proof}

The following theorem provides the $3$-dissection of $CP_3(n)$. 

\begin{theorem}\label{thm33} We have
\begin{align}
\sum_{n=0}^\infty CP_3(3n)q^n &= \dfrac{f_2^2f_3^{18}}{f_1^6f_6^6}+2q\dfrac{f_3^9f_6^3}{f_1^3f_2}+28q^2\dfrac{f_6^{12}}{f_2^4},\label{eq17}\\
\sum_{n=0}^\infty CP_3(3n+1)q^n &= 2\dfrac{f_2f_3^{15}}{f_1^5f_6^3}+13q\dfrac{f_3^6f_6^6}{f_1^2f_2^2}-16q^2\dfrac{f_1f_6^{15}}{f_2^5f_3^3},\label{eq18}\\
\sum_{n=0}^\infty CP_3(3n+2)q^n &= 7\dfrac{f_3^{12}}{f_1^4}-4q\dfrac{f_3^3f_6^9}{f_1f_2^3}+16q^2\dfrac{f_1^2f_6^{18}}{f_2^6f_3^6}.\label{eq19}
\end{align}
\end{theorem}

\begin{proof}
Substituting (\ref{eq7}) into (\ref{eq12}) yields
\begin{align*}
\sum_{n=0}^\infty CP_3(n)q^n = \left(\dfrac{f_9^9}{f_3^6f_6^2f_{18}^3}+q\dfrac{f_9^6}{f_3^5f_6^3}+3q^2\dfrac{f_9^3f_{18}^3}{f_3^4f_6^4}-2q^3\dfrac{f_{18}^6}{f_3^5f_6^5}+4q^4\dfrac{f_{18}^9}{f_3^2f_6^6f_9^3}\right)^2\dfrac{f_3^6}{f_6^6}.
\end{align*}
Extracting the terms with exponents congruent to $0, 1$ and $2$ modulo $3$ gives
\begin{align*}
\sum_{n=0}^\infty CP_3(3n)q^{3n} &= \dfrac{f_6^2f_9^{18}}{f_3^6f_{18}^6}+2q^3\dfrac{f_9^9f_{18}^3}{f_3^3f_6}+28q^6\dfrac{f_{18}^{12}}{f_6^4},\\
\sum_{n=0}^\infty CP_3(3n+1)q^{3n+1} &= 2q\dfrac{f_6f_9^{15}}{f_3^5f_{18}^3}+13q^4\dfrac{f_9^6f_{18}^6}{f_3^2f_6^2}-16q^7\dfrac{f_3f_{18}^{15}}{f_6^5f_9^3},\\
\sum_{n=0}^\infty CP_3(3n+2)q^{3n+2} &= 7q^2\dfrac{f_9^{12}}{f_3^4}-4q^5\dfrac{f_9^3f_{18}^9}{f_3f_6^3}+16q^8\dfrac{f_3^2f_{18}^{18}}{f_6^6f_{18}^6}.
\end{align*}
We then divide these three identities by $1, q$ and $q^2$, respectively, and replace $q^3$ with $q$ to arrive at (\ref{eq17}), (\ref{eq18}) and (\ref{eq19}).
\end{proof}

As a byproduct of Theorem \ref{thm33}, we find an infinite family of internal congruences modulo $2$.

\begin{corollary}\label{cor34} 
We have 
\begin{align}
CP_3(3^kn+3^k-2)\equiv CP_3(n-1)\pmod{2} \label{eq20}
\end{align} 
for all $k, n\geq 1$.
\end{corollary}

\begin{proof}
Reducing (\ref{eq18}) modulo $2$ yields
\begin{align*}
\sum_{n=0}^\infty CP_3(3n+1)q^n &\equiv q\dfrac{f_3^6f_6^6}{f_1^2f_2^2}\equiv \sum_{n=0}^\infty CP_3(n)q^{n+1}\pmod{2}.\\
\end{align*}
Since $CP_3(1)=2$, we have 
\begin{align*}
\sum_{n=1}^\infty CP_3(3n+1)q^n &\equiv \sum_{n=1}^\infty CP_3(n-1)q^n\pmod{2}\\
\end{align*}
so that $CP_3(3n+1)\equiv CP_3(n-1)\pmod{2}$, which is (\ref{eq20}) for $k=1$. Hence, the congruence (\ref{eq20}) for $k\geq 2$ now follows by induction.
\end{proof}

We next provide the $3$-dissection formulas for $CP_3(2n+1)$.

\begin{theorem}\label{thm35} We have
\begin{align}
\sum_{n=0}^\infty CP_3(6n+1)q^n &= 2\dfrac{f_2^6f_3^{12}}{f_1^8f_6^6}+32q\dfrac{f_2^7f_3^3f_6^3}{f_1^5},\label{eq21}\\
\sum_{n=0}^\infty CP_3(6n+3)q^n &= 8\dfrac{f_2^9f_3^9}{f_1^7f_6^3}+32q\dfrac{f_2^6f_6^6}{f_1^4},\label{eq22}\\
\sum_{n=0}^\infty CP_3(6n+5)q^n &= 24\dfrac{f_2^8f_3^6}{f_1^6}.\label{eq23}
\end{align}
\end{theorem}

\begin{proof}
Substituting (\ref{eq8}) into (\ref{eq14}) yields
\begin{align*}
\sum_{n=0}^\infty CP_3(2n+1)q^n &=2\left(\dfrac{f_6^4f_9^6}{f_3^8f_{18}^3}+2q\dfrac{f_6^3f_9^3}{f_3^7}+4q^2\dfrac{f_6^2f_{18}^3}{f_3^6}\right)^2f_3^8f_6^2.
\end{align*}
Extracting the terms with exponents congruent to $0, 1$ and $2$ modulo $3$ gives
\begin{align*}
\sum_{n=0}^\infty CP_3(6n+1)q^{3n} &= 2\dfrac{f_6^6f_9^{12}}{f_3^8f_{18}^6}+32q^3\dfrac{f_6^7f_9^3f_{18}^3}{f_3^5},\\
\sum_{n=0}^\infty CP_3(6n+3)q^{3n+1} &= 8q\dfrac{f_6^9f_9^9}{f_3^7f_{18}^3}+32q^4\dfrac{f_6^6f_{18}^6}{f_3^4},\\
\sum_{n=0}^\infty CP_3(6n+5)q^{3n+2} &= 24q^2\dfrac{f_6^8f_9^6}{f_3^6}.
\end{align*}
We then divide these three identities by $1, q$ and $q^2$, respectively, and replace $q^3$ with $q$ to arrive at (\ref{eq21}), (\ref{eq22}) and (\ref{eq23}).
\end{proof}

We now enumerate congruences for $CP_3(24n+r)$ for some values of nonnegative integers $r\leq 23$.

\begin{theorem}\label{thm36}We have
\begin{align*}
CP_3(24n+7)&\equiv 0\pmod{16},\\
CP_3(24n+13)&\equiv 0\pmod{4},\\
CP_3(24n+19)&\equiv 0\pmod{8}.
\end{align*}
\end{theorem}

\begin{proof}
Substituting (\ref{eq3}) and (\ref{eq5}) into (\ref{eq21}) and reducing modulo $32$, we get
\begin{align}
\sum_{n=0}^\infty CP_3(6n+1)q^n&\equiv 2\left(\dfrac{f_4^4f_6f_{12}^2}{f_2^5f_8f_{24}}+2q\dfrac{f_4f_6^2f_8f_{24}}{f_2^4f_{12}}\right)^6\left(\dfrac{f_4^{10}}{f_2^2f_8^4}-4q\dfrac{f_2^2f_8^4}{f_4^2}\right)\dfrac{f_2^6}{f_6^6}\pmod{32}.\label{eq24}
\end{align}
Extracting the terms in (\ref{eq24}) with even exponents and replacing $q^2$ with $q$ yields
\begin{align}
\sum_{n=0}^\infty CP_3(12n+1)q^n&\equiv 2\dfrac{f_2^{34}f_6^{12}}{f_1^{26}f_4^{10}f_{12}^6}\pmod{8}.\label{eq25}
\end{align}
We next put (\ref{eq2}) in (\ref{eq25}), extract the terms with odd exponents, divide by $q$, and replace $q^2$ with $q$, obtaining
\begin{align*}
\sum_{n=0}^\infty CP_3(24n+13)q^n&\equiv 4\dfrac{f_3^{12}f_4^{59}}{f_1^{31}f_2^8f_6^6f_8^{22}}\pmod{8},
\end{align*}
so that $CP_3(24n+13)\equiv 0\pmod{4}$. On the other hand, we consider the terms in (\ref{eq24}) with odd exponents, divide by $q$, and replace $q^2$ with $q$ to get
\begin{align}
\sum_{n=0}^\infty CP_3(12n+7)q^n&\equiv -8\dfrac{f_2^{22}f_6^{12}}{f_1^{22}f_4^2f_{12}^6}-8\dfrac{f_2^{31}f_3f_6^9}{f_1^{25}f_4^8f_{12}^4}\pmod{32}.\label{eq26}
\end{align}
Substituting (\ref{eq2}) and (\ref{eq6}) into (\ref{eq26}) yields
\begin{align}
\sum_{n=0}^\infty CP_3(12n+7)q^n&\equiv -8\left(\dfrac{f_8^5}{f_2^5f_{16}^2}+2q\dfrac{f_4^2f_{16}^2}{f_2^5f_8}\right)^{11}\dfrac{f_2^{22}f_6^{12}}{f_4^2f_{12}^6}\nonumber\\
&-8\left(\dfrac{f_4^6f_6^3}{f_2^9f_{12}^2}+3q\dfrac{f_4^2f_6f_{12}^2}{f_2^7}\right)\left(\dfrac{f_8^5}{f_2^5f_{16}^2}+2q\dfrac{f_4^2f_{16}^2}{f_2^5f_8}\right)^{11}\dfrac{f_2^{31}f_6^9}{f_4^8f_{12}^4}\pmod{32}.\label{eq27}
\end{align}
We consider the terms in (\ref{eq27}) with even exponents and replace $q^2$ with $q$ to arrive at
\begin{align*}
\sum_{n=0}^\infty CP_3(24n+7)q^n&\equiv -16\dfrac{f_3^{12}f_4^{55}}{f_1^{33}f_2^2f_6^6f_8^{22}}\pmod{32},
\end{align*}
which follows that $CP_3(24n+7)\equiv 0\pmod{16}$. On the other hand, we consider the terms in (\ref{eq27}) with odd exponents, divide by $q$, and replace $q^2$ with $q$ to get
\begin{align*}
\sum_{n=0}^\infty CP_3(24n+19)q^n&\equiv 8\dfrac{f_3^{10}f_4^{55}}{f_1^{31}f_2^6f_6^2f_8^{22}}-8\dfrac{f_3^{12}f_4^{49}}{f_1^{33}f_6^6f_8^{18}}\pmod{32}, 
\end{align*}
so that $CP_3(24n+19)\equiv 0\pmod{8}$.
\end{proof}

\begin{theorem}\label{thm37} We have
\begin{align*}
CP_3(24n+15)&\equiv 0\pmod{16},\\
CP_3(24n+21)&\equiv 0\pmod{16}.
\end{align*}
\end{theorem}

\begin{proof}
Putting (\ref{eq3}), (\ref{eq5}) and (\ref{eq7}) in (\ref{eq22}) and reducing modulo $32$, we see that the right-hand side of (\ref{eq22}) is congruent to 
\begin{align}
8\left(\dfrac{f_4^4f_6f_{12}^2}{f_2^5f_8f_{24}}+2q\dfrac{f_4f_6^2f_8f_{24}}{f_2^4f_{12}}\right)^5\left(\dfrac{f_8^2f_{12}^5}{f_2^2f_4f_6^4f_{24}^2}+q\dfrac{f_4^5f_{24}^2}{f_2^4f_6^2f_8^2f_{12}}\right)\left(\dfrac{f_4^{10}}{f_2^2f_8^4}-4q\dfrac{f_2^2f_8^4}{f_4^2}\right)\dfrac{f_2^9}{f_6^3}\label{eq28}\pmod{{32}}.
\end{align}
Extracting the terms in (\ref{eq28}) with even exponents and replacing $q^2$ with $q$ yields
\begin{align}
\sum_{n=0}^\infty CP_3(12n+3)q^n&\equiv 8\dfrac{f_2^{29}f_6^{15}}{f_1^{20}f_3^2f_8^7f_{12}^7}-16q\dfrac{f_2^{32}f_3f_6^6}{f_1^{21}f_4^9f_{12}}\pmod{32}.\label{eq29}
\end{align}
We next apply (\ref{eq2}), (\ref{eq6}) and (\ref{eq7}) into (\ref{eq29}) to arrive at
\begin{align}
\sum_{n=0}^\infty CP_3(12n+3)q^n&\equiv 8\left(\dfrac{f_8^5}{f_2^5f_{16}^2}+2q\dfrac{f_4^2f_{16}^2}{f_2^5f_8}\right)^9\left(\dfrac{f_8^2f_{12}^5}{f_2^2f_4f_6^4f_{24}^2}+q\dfrac{f_4^5f_{24}^2}{f_2^4f_6^2f_8^2f_{12}}\right)\dfrac{f_2^{29}f_6^{15}}{f_4^7f_{12}^7}\nonumber\\
&-16q\left(\dfrac{f_8^5}{f_2^5f_{16}^2}+2q\dfrac{f_4^2f_{16}^2}{f_2^5f_8}\right)^9\left(\dfrac{f_4^6f_6^3}{f_2^9f_{12}^2}+3q\dfrac{f_4^2f_6f_{12}^2}{f_2^7}\right)\dfrac{f_2^{32}f_6^6}{f_4^9f_{12}}\pmod{32}.\label{eq30}
\end{align}
We consider the terms in (\ref{eq30}) with odd exponents, divide by $q$, and replace $q^2$ with $q$ to get
\begin{align*}
\sum_{n=0}^\infty CP_3(24n+15)q^n&\equiv 16\dfrac{f_3^7f_4^{43}f_6^3}{f_1^{20}f_2^7f_8^{14}f_{12}^4}+16q\dfrac{f_2^5f_3^{11}f_4^{35}f_{12}^4}{f_1^{24}f_6^9f_8^{14}}\pmod{32}, 
\end{align*}
so that $CP_3(24n+15)\equiv 0\pmod{16}$. On the other hand, we look for the terms in (\ref{eq28}) with odd exponents, divide by $q$, and replace $q^2$ with $q$ so that
\begin{align}
\sum_{n=0}^\infty CP_3(12n+9)q^n&\equiv 8\dfrac{f_2^{35}f_6^9}{f_1^{22}f_4^{11}f_{12}^3}-16\dfrac{f_2^{26}f_6^{12}}{f_1^{19}f_3f_4^5f_{12}^5}\pmod{32}.\label{eq31}
\end{align}
We now apply (\ref{eq2}) and (\ref{eq7}) into (\ref{eq31}) to get
\begin{align}
\sum_{n=0}^\infty CP_3(12n+9)q^n&\equiv 8\left(\dfrac{f_8^5}{f_2^5f_{16}^2}+2q\dfrac{f_4^2f_{16}^2}{f_2^5f_8}\right)^{11}\dfrac{f_2^{35}f_6^9}{f_4^{11}f_{12}^3}\nonumber\\
&-16\left(\dfrac{f_8^5}{f_2^5f_{16}^2}+2q\dfrac{f_4^2f_{16}^2}{f_2^5f_8}\right)^9\left(\dfrac{f_8^2f_{12}^5}{f_2^2f_4f_6^4f_{24}^2}+q\dfrac{f_4^5f_{24}^2}{f_2^4f_6^2f_8^2f_{12}}\right)\dfrac{f_2^{26}f_6^{12}}{f_4^5f_{12}^5}\pmod{32}.\label{eq32}
\end{align}
We extract the terms in (\ref{eq32}) with odd exponents, divide by $q$, and replace $q^2$ with $q$ so that
\begin{align*}
\sum_{n=0}^\infty CP_3(24n+21)q^n&\equiv 16\dfrac{f_3^9f_4^{49}}{f_1^{20}f_2^9f_6^3f_8^{18}}-16\dfrac{f_3^{10}f_4^{43}f_{12}^2}{f_1^{23}f_6^6f_8^{18}}\pmod{32}, 
\end{align*}
which yields $CP_3(24n+21)\equiv 0\pmod{16}$.
\end{proof}

\begin{theorem}\label{thm38} We have
\begin{align*}
CP_3(24n+11)&\equiv 0\pmod{48},\\
CP_3(24n+23)&\equiv 0\pmod{96}.
\end{align*}
\end{theorem}

\begin{proof}
We start with putting (\ref{eq5}) into (\ref{eq23}), obtaining
\begin{align}
\sum_{n=0}^\infty CP_3(6n+5)q^n=24\left(\dfrac{f_4^4f_6f_{12}^2}{f_2^5f_8f_{24}}+2q\dfrac{f_4f_6^2f_8f_{24}}{f_2^4f_{12}}\right)^3f_2^8.\label{eq33}
\end{align}
Considering the terms in (\ref{eq33}) with odd exponents, we divide them by $q$ and replace $q^2$ with $q$ so that
\begin{align}
\sum_{n=0}^\infty CP_3(12n+11)q^n=144\dfrac{f_2^9f_3^4f_6^3}{f_1^6f_4f_{12}}+192q\dfrac{f_2^3f_3^6f_4^3f_{12}^3}{f_1^4f_6^3}.\label{eq34}
\end{align}
We now apply (\ref{eq2}), (\ref{eq4}) and (\ref{eq5}) into (\ref{eq34}) to get
\begin{align}
\sum_{n=0}^\infty CP_3(12n+11)q^n&=144\left(\dfrac{f_4^4f_6f_{12}^2}{f_2^5f_8f_{24}}+2q\dfrac{f_4f_6^2f_8f_{24}}{f_2^4f_{12}}\right)^2\left(\dfrac{f_8^5}{f_2^5f_{16}^2}+2q\dfrac{f_4^2f_{16}^2}{f_2^5f_8}\right)\dfrac{f_2^9f_6^3}{f_4f_{12}}\nonumber\\
&+192q\left(\dfrac{f_4^3f_6^2}{f_2^2f_{12}}+q\dfrac{f_{12}^3}{f_4}\right)^2\left(\dfrac{f_8^5}{f_2^5f_{16}^2}+2q\dfrac{f_4^2f_{16}^2}{f_2^5f_8}\right)\dfrac{f_2^3f_4^3f_{12}^3}{f_6^3}.\label{eq35}
\end{align}
Extracting the terms in (\ref{eq35}) with even exponents and replacing $q^2$ with $q$ gives
\begin{align*}
\sum_{n=0}^\infty CP_3(24n+11)q^n&\equiv 48\dfrac{f_2^7f_3^5f_4^3f_6^3}{f_1^6f_8^2f_{12}^2}\pmod{96},
\end{align*}
which follows that $CP_3(24n+11)\equiv 0\pmod{48}$. On the other hand, we look for the terms in (\ref{eq35}) with odd exponents, divide by $q$, and replace $q^2$ with $q$ so that
\begin{align*}
\sum_{n=0}^\infty CP_3(24n+23)q^n&\equiv 96\dfrac{f_2^9f_3^5f_6^3f_8^2}{f_1^6f_4^3f_{12}^2}\pmod{192},
\end{align*}
so that $CP_3(24n+23)\equiv 0\pmod{96}$.
\end{proof}

We end this section by presenting an infinite family of congruences modulo twice times powers of $3$, as the following result says.

\begin{theorem}\label{thm39} We have
\begin{align*}
CP_3(2\cdot 3^kn+3^k-2)\equiv 0\pmod{2\cdot 3^{k-1}}
\end{align*}
for all $k, n\geq 1$.
\end{theorem}

\begin{proof}
In view of (\ref{eq14}) and Corollary \ref{cor34}, it suffices to prove that
\begin{align}
CP_3(3^kn+3^k-2)\equiv 0\pmod{3^{k-1}}\label{eq36}
\end{align}
for all $k, n\geq 1$. To show (\ref{eq36}), we follow the argument used to prove \cite[Thm. 3.1]{giree}. Let $a = q^{-1}f_1f_2ff_9^{-1}f_{18}^{-1}, b=q^{-1}x(q^3)^{-1}$ and $c = q^{-3}f_3^4f_6^4f_9^{-4}f_{18}^{-4}$. Then by (\ref{eq10}) and (\ref{eq11}) of Lemma \ref{lem23}, we see that
\begin{align}
a &= b-1-\dfrac{2}{b},\label{eq37}\\
c &= b^3-7-\dfrac{8}{b^3}.\label{eq38}
\end{align}
From (\ref{eq37}) and (\ref{eq38}), we note that
\begin{align*}
c+7 = \left(b-\dfrac{2}{b}\right)^3+6\left(b-\dfrac{2}{b}\right) = (a+1)^3+6(a+1),
\end{align*}
which shows that $c = a^3+3a^2+9a$ and 
\begin{align}
\dfrac{1}{a^2}=\dfrac{1}{c^2}(a^4+6a^3+27a^2+54a+81).\label{eq39}
\end{align}
We rewrite (\ref{eq12}) as
\begin{align}
\sum_{n=0}^\infty CP_3(n)q^{n-1} =\dfrac{f_3^6f_6^6}{q^3f_9^2f_{18}^2}\cdot\dfrac{1}{a^2},\label{eq40}
\end{align}
and define the \textquotedblleft huffing" operator $H_3$ modulo $3$ (which is a linear map) by 
\begin{align*}
H_3\left(\sum_{n=0}^\infty a(n)q^n\right) = \sum_{n=0}^\infty a(3n)q^{3n}.
\end{align*}
Using (\ref{eq37}), we compute
\begin{align*}
H_3(1) &= 1,\\
H_3(a) &= -1,\\
H_3(a^2) &= -3,\\
H_3(a^3) &= -\dfrac{8}{b^3}+11+b^3,\\
H_3(a^4) &= \dfrac{32}{b^3}+1-4b^3,
\end{align*}
so that from (\ref{eq38}) and (\ref{eq39}), we get
\begin{align}
H_3\left(\dfrac{1}{a^2}\right)&=\dfrac{1}{c^2}\left(\dfrac{32}{b^3}+1-4b^3-\dfrac{48}{b^3}+66+6b^3-81-54+81\right)\nonumber\\
&= \dfrac{1}{c^2}\left(-\dfrac{16}{b^3}+13+2b^3\right) = \dfrac{2}{c}+\dfrac{27}{c^2}.\label{eq41}
\end{align}
Thus, applying $H_3$ to (\ref{eq40}), using (\ref{eq41}), and replacing $q^3$ with $q$ lead to 
\begin{align}
\sum_{n=0}^\infty CP_3(3n+1)q^n =2(f_1f_2f_3f_6)^2+27q\dfrac{f_3^6f_6^6}{f_1^2f_2^2}.\label{eq42}
\end{align}
We now define 
\begin{align}
(f_1f_2f_3f_6)^2 = q^2(f_3f_6f_9f_{18})^2a^2 := \sum_{n=0}^\infty d(n)q^n,\label{eq43}
\end{align}
so that, in view of (\ref{eq42}), we have
\begin{align}
CP_3(3n+1) = 2d(n)+27CP_3(n-1)\label{eq44}
\end{align}
for all $n\geq 1$. Applying $H_3$ to (\ref{eq43}) and replacing $q^3$ with $q$ lead to 
\begin{align*}
\sum_{n=0}^\infty d(3n+2)q^n=-3(f_1f_2f_3f_6)^2,
\end{align*}
which means that
\begin{align}
d(3n+2) = -3d(n)\label{eq45}
\end{align}
for all $n\geq 0$. Thus, from (\ref{eq44}) and (\ref{eq45}), we get
\begin{align*}
CP_3(9n+7) &= 2d(3n+2)+27CP_3(3n+1)=-6d(n)+27(2d(n)+27CP_3(n-1))\\
&=48d(n)+729CP_3(n-1)
\end{align*}
and by induction, we finally arrive at
\begin{align*}
CP_3(3^kn+3^k-2) = 3^{k-1}\cdot\dfrac{9^k-(-1)^k}{5}\cdot d(n)+3^{3k}CP_3(n-1)
\end{align*}
for all $k,n\geq 1$. Since $9^k-(-1)^k$ is divisible by $5$ for all $k\geq 1$, we immediately obtain (\ref{eq36}).
\end{proof}


\begin{thebibliography}{99}
\bibitem{ahlg} S. Ahlgren, {\it Distribution of the partition function modulo composite integers $M$}, Math. Ann. {\bf 318} (2000), 795--803.	

\bibitem{anhrse} G. E. Andrews, M. D. Hirschhorn and J. A. Sellers, {\it Arithmetic properties of partitions with even parts distinct}, Ramanujan J. {\bf 23} (2010), 169--181.	

\bibitem{berndt} B. C. Berndt, {\it Ramanujan's Notebooks, Part III}, Springer, New York, 1991.

\bibitem{chan} H.-C. Chan, {\it Ramanujan's cubic continued fraction and an analog of his 'Most beautiful identity'}, Int. J. Number Theory {\bf 6} (2010), 673--680.
	
\bibitem{chan2} H.-C. Chan, {\it Ramanujan's cubic continued fraction and Ramanujan type congruences for a certain partition function}, Int. J. Number Theory {\bf 6} (2010), 819--834.

\bibitem{chen} S. Chen, {\it Arithmetical properties of the number of $t$-core partitions}, Ramanujan J. {\bf 18} (2009), 103--112.

\bibitem{chen2} S. Chen, {\it Congruences for $t$-core partition functions}, J. Number Theory {\bf 133} (2013), 4036--4046.

\bibitem{chernda} S. Chern and M. G. Dastidar, {\it Congruences and recursions for the cubic partition}, Ramanujan J. {\bf 44} (2017), 559--556.

\bibitem{cuigu} S. P. Cui and N. S. S. Gu, {\it Arithmetic properties of $l$-regular partitions}, Adv. Appl. Math. {\bf 51} (2013), 507--523.

\bibitem{dasilse} R. da Silva and J. A. Sellers, {\it Congruences for $3$-core cubic partitions}, Indian J. Pure Appl. Math. {\bf 54} (2023), 404--420.

\bibitem{giree} D. S. Gireesh, {\it Formulas for cubic partition with $3$-cores}, J. Math. Anal. Appl. {\bf 453} (2017), 20--31.

\bibitem{hirsc} M. D. Hirschhorn, {\it The Power of $q$, A Personal Journey}, Developments in Mathematics vol. 49, Springer, New York, 2017.

\bibitem{naika} M. S. Mahadeva Naika and B. Hemanthkumar, {\it Arithmetic properties of $5$-regular bipartitions}, Int. J. Number Theory {\bf 13} (2017), 939--956.

\bibitem{raman} S. Ramanujan, {\it Some properties of $p(n)$, the number of partitions of $n$}, Proc. Cambridge Philos. Soc. {\bf 19} (1919), 207--210.

\bibitem{raman2} S. Ramanujan, {\it Congruence properties of partitions}, Math. Z. {\bf 9} (1921), 147--153.

\bibitem{wang} L. Wang, {\it Arithmetic identities and congruence for partition triples with $3$-cores}, Int. J. Number Theory {\bf 12} (2016), 995--1010.

\bibitem{wang2} L. Wang, {\it Congruences modulo powers of $5$ for two restricted bipartitions}, Ramanujan J. {\bf 44} (2017), 471--4491.
\end{thebibliography}
\end{document}